\documentclass[a4paper,11pt]{article}
\usepackage{MyStyle}
\usepackage{aromatic_forms}

\usepackage{hyperref}

\usepackage{algorithm}
\floatname{algorithm}{}
\usepackage{algorithmic}

\newcommand{\Liehomotopy}{H}

\usepackage{circledsteps}

\allowdisplaybreaks

\newtheorem{theorem}{Theorem}[section]
\newtheorem{definition}[theorem]{Definition}
\newtheorem*{definition*}{Definition}

\newtheorem{proposition}[theorem]{Proposition}

\newtheorem{remark}[theorem]{Remark}
\newtheorem*{remark*}{Remark}

\newtheorem*{remarks*}{Remarks}

\newtheorem*{notation*}{Notation}

\newtheorem*{ex*}{Example}

\newtheorem*{exs*}{Examples}

\newtheorem*{app*}{Application}
\newtheorem{conjecture*}{Conjecture}

\def\ts{\thinspace}

\title{
The Lie derivative and Noether's theorem on the aromatic bicomplex for the study of volume-preserving numerical integrators
}

\author{
Adrien Laurent\textsuperscript{1}
}

\begin{document}
\footnotetext[1]{
Univ Rennes, INRIA (Research team MINGuS), IRMAR (CNRS UMR 6625) and ENS Rennes,
France.
Adrien.Laurent@INRIA.fr.}

\maketitle

\begin{abstract}
The aromatic bicomplex is an algebraic tool based on aromatic Butcher trees and used in particular for the explicit description of volume-preserving affine-equivariant numerical integrators.
The present work defines new tools inspired from variational calculus such as the Lie derivative, different concepts of symmetries, and Noether's theory in the context of aromatic forests.
The approach allows to draw a correspondence between aromatic volume-preserving methods and symmetries on the Euler-Lagrange complex, to write Noether's theorem in the aromatic context, and to describe the aromatic B-series of volume-preserving methods explicitly with the Lie derivative.


\smallskip

\smallskip

\end{abstract}


\section{Introduction}

The search for an affine-equivariant volume-preserving method is a central open problem of geometric numerical integration.
Such an integrator takes the form of an aromatic Butcher-series method~\cite{MuntheKaas16abs}.
While Butcher-series describe the Taylor expansion of the flow of ordinary differential equations and of a large class of their numerical approximations~\cite{Butcher72aat,Hairer74otb} (see also the textbooks~\cite{Hairer06gni,Butcher16nmf,Butcher21bsa} and the review~\cite{McLachlan17bsa}), aromatic B-series were introduced in~\cite{Iserles07bsm,Chartier07pfi} specifically for the study of volume preservation (see also~\cite{Kang95vpa, McLachlan16bsm, MuntheKaas16abs, Bogfjellmo19aso, Floystad20tup, Bogfjellmo22uat, Lejay22cgr}) as B-series methods cannot preserve volume in general.
We mention that finding a volume-preserving aromatic B-series method is the first step toward the creation of an exotic aromatic S-series method~\cite{Laurent20eab, Laurent21ocf, Laurent21ata,  Bronasco22ebs, Laurent23tue} that exactly preserves the invariant measure of ergodic stochastic differential equations as the algebraic conditions are similar (see, for instance,~\cite{Abdulle15lta}).

The recent work~\cite{Laurent23tab} introduces new tools from the calculus of variations, such as the aromatic bicomplex (see also~\cite{Anderson89tvb,Anderson92itt,Olver93aol,Mansfield10apg} and references therein), and these tools yield valuable insight of the form of the Taylor expansion of a volume-preserving methods.
In particular, it shows that aromatic Runge-Kutta methods do not preserve volume in general, while aromatic exponential methods are promising starting points.
To further understand the form of a volume-preserving method, the present work defines new tools on the aromatic bicomplex such as the Lie derivative, different concepts of symmetries, and Noether's theory in the aromatic context.

Let~$\tau$,~$\gamma$ be linear combinations of trees, the Lie derivative of~$\gamma$ in the direction of~$\tau$ is
\begin{equation}
\label{equation:def_perturbation_Lie_der}
\LL_\tau \gamma =\frac{d}{d\varepsilon}\Big|_{\varepsilon=0}[(\atree1101+\varepsilon \tau)\triangleright \gamma],
\end{equation}
where~$\triangleright$ is the substitution law of B-series~\cite{Hairer06gni,Chartier10aso,Calaque11tih} (see also~\cite{Bogfjellmo19aso,Bronasco22cef} for the substitution of aromatic and exotic aromatic series).
The Lie derivative is also the pre-Lie version of the substitution law on B-series. It appears under the name pre-Lie insertion product in~\cite{Saidi10oap,Manchon11lpl,Saidi11adh} for the study of the freeness of the pre-Lie insertion algebra.
In the calculus of variations, the Lie-derivative defines symmetries, that are perturbations that leave the input unchanged at first order. This leads to the Noether theorem, that draws links between symmetries and conservation laws.

This paper gives a general definition of the Lie derivative on aromatic forms by using the framework given by the aromatic bicomplex. This allows us to define symmetries, to write an aromatic version of the Noether theorem, and to draw further links between variational calculus and numerical volume-preservation.
We give a concise introduction in Section~\ref{section:preliminaries} of the aromatic bicomplex and its properties, while Section~\ref{section:Noether} is devoted to the general definition of the Lie derivative, of the different symmetries, and the statement of the aromatic Noether theorem.
We then adapt the new approach and results to the study of volume preserving methods.


%
%



\section{Preliminaries on the aromatic bicomplex}
\label{section:preliminaries}

In this section, we give a concise definition of the necessary tools and concepts required for the definition of the aromatic Euler-Lagrange complex. The notations and vocabulary are chosen to match with the literature of variational calculus, as we shall draw bridges between numerical analysis and variational calculus in Section~\ref{section:Noether}. We skip the technical details on the Euler and homotopy operators and refer the reader to~\cite{Laurent23tab} for more details.

\subsection{Aromatic forms and their derivatives}

While aromatic trees represent vector fields, we use aromatic forests to represent specific classes of homogeneous tensors and forms. This allows in particular to translate the technicalities of the infinite jet bundle~$J^\infty(\R^d)$ into straightforward combinatorics.
\begin{definition}
Let~$V$ be a finite set of nodes, that we split into vertices~$V^{\bullet}$ and covertices~$V^{\circ}$, and~$E\subset V\times V$ a set of oriented edges.
The covertices are numbered from~$1$ to~$p$, while the vertices are indistinguishable.
Each node in~$V$ is the source of exactly one edge, except the roots that have no outgoing edges, that we order and number from~$1$ to~$n$.
Any connected component of such a graph either has exactly one root, and is called a tree, or does not have a root, and is called an aroma.
We call aromatic forests such graphs, up to equivalence of graphs that preserve the numbering of the covertices and the roots. We write~$\FF_{n,p}$ the set of aromatic forests with~$n$ roots and~$p$ covertices and~$\FF_{n}=\FF_{n,0}$.
The number of nodes~$\abs{\gamma}$ of an aromatic forest~$\gamma$ is called the order of~$\gamma$.
\end{definition}

In the spirit of differential geometry, we alternatize aromatic forests to obtain aromatic forms using the wedge projection operator.
\begin{definition}
For~$\gamma\in \FF_{n,p}$, let~$\SS_n^{\bullet}$ (resp.\ts~$\SS_p^{\circ}$) be the set of permutations of the roots of~$\gamma$ (resp.\ts the covertices of~$\gamma$). The root and covertex wedges of~$\gamma$ are
$$
\wedge^{\bullet} \gamma = \frac{1}{n!} \sum_{\sigma\in \SS_{n}^{\bullet}} \varepsilon(\sigma) \sigma \gamma, \quad \wedge^{\circ} \gamma = \frac{1}{p!} \sum_{\sigma\in \SS_p^{\circ}} \varepsilon(\sigma) \sigma \gamma,
$$
where~$\varepsilon(\sigma)$ is the signature of the permutation~$\sigma$.
The wedge is~$\wedge=\wedge^{\bullet}\wedge^{\circ}=\wedge^{\circ}\wedge^{\bullet}$.
We extend the wedge on~$\Span(\FF_{n,p})$ by linearity and we gather aromatic forms in~$\Omega_{n,p}=\wedge\Span(\FF_{n,p})$ and~$\Omega_n=\Omega_{n,0}$.
\end{definition}

\begin{ex*}
Let~$\gamma_1=\atree1101 \atree2101\in \FF_2$,~$\gamma_2=\atree2023\in \FF_{0,2}$,~$\gamma_3=\atree1111\atree1121\in \FF_{2,2}$, then
\[
\wedge \gamma_1=\frac{1}{2}(\atree1101 \atree2101-\atree2101 \atree1101) \in \Omega_2, \quad
\wedge \gamma_2=0, \quad
\wedge \gamma_3=\frac{1}{2}(\atree1111\atree1121-\atree1121\atree1111).
\]
\end{ex*}

The operations on the variational bicomplex make use of differentiations in the infinite jet bundle. In the aromatic context, we replace the differentiations by the operations of grafting and replacing nodes.
The horizontal derivative is defined using grafting operations, while the vertical derivative uses the replacing operation. The sign change in the definition of the total derivative is explained in~\cite[Rk.\ts 2.8]{Laurent23tab}. We mention that the horizontal and vertical derivatives were used in a different context on~$\Omega_1=\Span(\FF_1)$ respectively in~\cite{Chartier07pfi,Iserles07bsm} for~$d_H$ and in~\cite{Floystad20tup} for~$d_V$.
\begin{definition}
Let~$\gamma\in \FF_{n,p}$,~$r$ a root of~$\gamma$, and~$u\in V$ (possibly equal to~$r$), then~$D^{r\rightarrow u} \gamma$ returns a copy of~$\gamma$ where the node~$r$ is now a predecessor of~$u$. The operator~$D^r \gamma=\sum_{u\in V} D^{r\rightarrow u} \gamma$ grafts~$r$ to all possible nodes.
Let~$\gamma\in \FF_{n,p}$ and~$v\in V^\bullet$, then~$\gamma_{v\rightarrow \Circled{k}}$ is the forest obtained by replacing the node~$v$ by a new covertex~$\Circled{k}$. Similarly,~$\gamma_{\Circled{k}\rightarrow \tau}$ is the linear combination of forests obtained by replacing the covertex~$\Circled{k}$ by the tree~$\tau$ and grafting the predecessors of~$\Circled{k}$ to the nodes of~$\tau$ in all possible ways.
The horizontal, vertical, and total derivatives of~$\gamma\in\FF_{n,p}$ are
$$d_H \gamma=D^{r_n}\gamma,\quad d_V \gamma=\wedge \sum_{v\in V^{\bullet}} \gamma_{v\rightarrow \Circled{p+1}}, \quad d\gamma=(-1)^{n+p}d_H\gamma+ d_V\gamma.$$
We extend~$d_H$ and~$d_V$ by linearity into~$d_H\colon \Omega_{n,p} \rightarrow \Omega_{n-1,p}$ and~$d_V\colon \Omega_{n,p} \rightarrow \Omega_{n,p+1}$, with the convention~$d_H \gamma=0$ if~$\gamma\in \Omega_{0,p}$.
\end{definition}


\begin{ex*}
Consider~$\gamma_1=\atree1101 \in \Omega_1$,~$\gamma_2=\wedge \atree1101 \atree2101\in \Omega_2$, and~$\gamma_3=\wedge \atree1101 \atree1111\in \Omega_{2,1}$, then, we find
\[d_H\gamma_1=\atree1001, \quad d_V\gamma_1=\atree1111,\]
\[
d_H\gamma_2=\frac{1}{2}\Big(\atree2002 \atree1101+\atree2001 \atree1101-\atree1001 \atree2101-\atree3101\Big), \quad
d_V\gamma_2=\wedge \atree1111 \atree2101+\wedge \atree1101 \atree2111+\wedge \atree1101 \atree2112
,\]
\[
d_H\gamma_3=\frac{1}{2}\Big(\atree1011 \atree1101+\atree2111-\atree1001 \atree1111-\atree2112\Big), \quad
d_V\gamma_3=\wedge\atree1121 \atree1111=\frac{1}{2}(\atree1121 \atree1111-\atree1111 \atree1121)
.\]
\end{ex*}

The derivatives on aromatic forms naturally form complexes as justified by the following result. Note that the horizontal and vertical derivatives commute, while their equivalents in variational calculus anticommute~\cite{Anderson92itt}.
\begin{proposition}[\cite{Laurent23tab}]
\label{proposition:derivatives_squared}
The derivatives satisfy
$$
d_H^2=0, \quad d_V^2=0, \quad d_V d_H=d_H d_V, \quad d^2=0.
$$
\end{proposition}

\subsection{The aromatic bicomplex and the Euler-Lagrange complex}
\label{section:complexes}

The object of ultimate interest in variational calculus is the Euler-Lagrange complex, which requires defining the augmented bicomplex first.
The interior Euler operator~$I$ is given for an aromatic form~$\gamma\in\Omega_{0,1}$ by the combination of forms obtained by unplugging all the predecessors of the covertex of~$\gamma$, plugging back the predecessors on all the vertices in all possible ways, and multiplying by~$-1$ if the number of predecessors is odd. For instance, we have
\[
I \atree1011=0, \quad
I \atree2013=\atree2013, \quad
I \atree2012=\atree2013, \quad
I \atree2011=-\atree2013, \quad
I \atree1001 \atree1011=-\atree2013.
\]
For the sake of simplicity, we refer to~\cite[Sect.\ts 4.1]{Laurent23tab} for the precise definition of~$I\colon \Omega_{0,p}\rightarrow\Omega_{0,p}$.
The variational derivative~$\delta_V=I\circ d_V$ and the interior Euler operator~$I$ satisfy
$$
I^2=I,\quad I d_H=0,\quad \delta_V^2=0.
$$
In particular,~$I$ is a projection on~$\II_p=I(\Omega_{0,p})$, the aromatic equivalent of the space of source forms.

The augmented aromatic bicomplex is the diagram drawn in Figure~\ref{figure:augmented_bicomplex} that displays the interactions between the different spaces of aromatic forms.
The aromatic bicomplex can be seen as a generalised subcomplex of the variational bicomplex~\cite{Anderson89tvb,Anderson92itt} that focuses on specific classes of homogeneous forms and is fully independent of the dimension of the problem (see~\cite[Rk.\ts 2.8]{Laurent23tab}).
\begin{figure}[!ht]
$$\begin{tikzcd}
 & \vdots & \vdots & \vdots & \vdots & \\
    \dots \arrow{r}{d_H} & \Omega_{2,2} \arrow{r}{d_H} \arrow{u}{d_V} & \Omega_{1,2} \arrow{r}{d_H} \arrow{u}{d_V} & \Omega_{0,2} \arrow{u}{d_V} \arrow{r}{I} & \II_2 \arrow{u}{\delta_V} \arrow{r} & 0\\
    \dots \arrow{r}{d_H} & \Omega_{2,1} \arrow{r}{d_H} \arrow{u}{d_V} & \Omega_{1,1} \arrow{r}{d_H} \arrow{u}{d_V} & \Omega_{0,1} \arrow{u}{d_V} \arrow{r}{I} & \II_1 \arrow{u}{\delta_V} \arrow{r} & 0\\
    \dots \arrow{r}{d_H} & \Omega_{2} \arrow{r}{d_H} \arrow{u}{d_V} & \Omega_{1} \arrow{r}{d_H} \arrow{u}{d_V} & \Omega_{0} \arrow{u}{d_V} \arrow{ru}{\delta_V} & & \\
     & 0 \arrow{u} & 0 \arrow{u} & 0 \arrow{u} & &
\end{tikzcd}$$
\caption{The augmented aromatic bicomplex.}
	\label{figure:augmented_bicomplex}
\end{figure}


The bottom row of the aromatic bicomplex is similar to the De Rham complex. It extends into the edge complex~\eqref{equation:Euler_Lagrange_complex}, called the aromatic Euler-Lagrange complex.
\begin{equation}
\label{equation:Euler_Lagrange_complex}
\begin{tikzcd}
\dots \arrow{r}{d_H} & \Omega_{2} \arrow{r}{d_H} & \Omega_{1} \arrow{r}{d_H} & \Omega_{0} \arrow{r}{\delta_V} & \II_1 \arrow{r}{\delta_V}& \II_2 \arrow{r}{\delta_V}& \dots
\end{tikzcd}
\end{equation}
We introduce some vocabulary to further motivate the importance of the aromatic Euler-Lagrange complex.
While we call aromatic forms the elements of~$\Omega_{n,p}$, the space~$\Omega_{1}$ spanned by aromatic trees represent both vector fields and differential forms in variational calculus, so that we shall call the elements of~$\Omega_{1}$ aromatic vector fields or aromatic forms depending on the context.
The elements of~$\Omega_{0}$ spanned by multi-aromas represent Lagrangians in variational calculus and volume forms in the context of volume-preservation, so that we call its elements aromatic Lagrangians.
The source forms in~$\II_1$ represent differential equations.
The Euler-Lagrange complex~\eqref{equation:Euler_Lagrange_complex} is the rigorous implementation of the following diagram (see~\cite{Anderson92itt}).
\[\begin{tikzcd}
\dots \arrow{r}{\text{curl}}
& \text{vector fields} \arrow{r}{\text{div}}
& \text{Lagrangians} \arrow{r}{\text{Euler-Lagrange}}
&[3.5em] \text{diff.\ts eq.} \arrow{r}{\text{Helmholtz}}
&[2.0em] \dots
\end{tikzcd}\]

The crucial property of the aromatic bicomplex is its exactness, that is, the kernel of a map in the complex coincides with the image of the preceding map.
For instance in the context of variational calculus~\cite{Anderson92itt}, a vector field is a curl if and only if it is divergence-free, and some differential equations are the Euler-Lagrange equations associated to a Lagrangian if and only if they are in the kernel of the Helmholtz operator\footnote{This characterization is often called the inverse problem for Lagrangian mechanics in the literature.}.
We refer the reader to~\cite{Laurent23tab} for the detailed proof of the exactness of the augmented aromatic bicomplex and for the explicit expressions of the associated homotopy operators.
\begin{theorem}[\cite{Laurent23tab}]
\label{theorem:exactness}
The horizontal and vertical sequences of the aromatic Euler-Lagrange complex~\eqref{equation:Euler_Lagrange_complex} and the augmented aromatic bicomplex are exact, that is, there exist homotopy operators
\[
h_H\colon \Omega_{n,p}\rightarrow \Omega_{n+1,p}, \quad
h_V\colon \Omega_{n,p}\rightarrow \Omega_{n,p-1}, \quad
\mathfrak{h}_H\colon \Omega_{0,p}\rightarrow \Omega_{1,p}, \quad
\mathfrak{h}_V\colon \II_p\rightarrow \II_{p-1},
\]
such that the following identities hold:
\begin{align}
\label{equation:horizontal_homotopy}
\gamma&=(d_H  h_H+h_H  d_H) \gamma, \quad \gamma \in \Omega_{n,p}, \quad n\geq 1, \quad p\geq 0,\\
\label{equation:vertical_homotopy}
\gamma&=(d_V  h_V+h_V  d_V) \gamma, \quad \gamma \in \Omega_{n,p}, \quad n\geq 0, \quad p\geq 1,\\
\gamma&=(d_H h_H +h_V \delta_V)\gamma, \quad \gamma \in \Omega_{0}, \nonumber\\
\label{equation:augmented_horizontal_homotopy}
\gamma&=(I+d_H \mathfrak{h}_H) \gamma, \quad \gamma \in \Omega_{0,p}, \quad p\geq 1,\\
\gamma&=(\delta_V  \mathfrak{h}_V+\mathfrak{h}_V  \delta_V) \gamma, \quad \gamma \in \II_p, \quad p\geq 1. \nonumber
\end{align}
\end{theorem}

In particular, if one is interested in the description of all of the aromatic vector fields~$\gamma\in \Omega_1$ of vanishing divergence~$d_H\gamma=0$, then Theorem~\ref{theorem:exactness} states that there exists an aromatic form~$\eta$ such that~$\gamma=d_H \eta\in \Img(d_H)$.
The problem of finding a volume-preserving numerical method translates via backward error analysis~\cite{Hairer06gni} into the precise description of~$\Ker(d_H|_{\Omega_1})$, so that the Euler-Lagrange complex becomes a strong tool in this context. We further describe~$\Ker(d_H|_{\Omega_1})$ using symmetries and Noether's theorem.

\section{Noether's theory on the aromatic bicomplex}
\label{section:Noether}

In this section, we define the Lie derivative of aromatic forms, study its properties, and use it to rewrite the Noether theorem in the context of the aromatic bicomplex. We then apply these new tools in the context of volume-preservation.

\subsection{The aromatic Lie derivative}

In the spirit of the differential geometry literature~\cite{Lee13its}, we use a Cartan formula to define the aromatic Lie derivative on aromatic forms. We then present the different properties of the Lie derivative.
\begin{definition}
For~$\gamma\in\FF_{n,p}$, the contraction of~$\gamma$ in direction of~$\tau\in \FF_1$ is
$$i_\tau \gamma=p\gamma_{\Circled{p}\rightarrow\tau},$$
where~$i_\tau\gamma=0$ if~$p=0$.
We extend the contraction in~$i_\tau\colon \Omega_{n,p}\rightarrow \Omega_{n,p-1}$ for~$\tau\in \Omega_{1}$ by linearity.
For an aromatic form~$\gamma\in \Omega_{n,p}$ and an aromatic vector field~$\tau \in \Omega_1$, the aromatic Lie derivative of~$\gamma$ in the direction of~$\tau$ is given by the Cartan formula:
\begin{equation}
\label{equation:Cartan}
\LL_\tau \gamma = (d i_\tau +i_\tau d)\gamma.
\end{equation}
\end{definition}

Thanks to the homotopy identities of Theorem~\ref{theorem:exactness}, the Lie derivative satisfies the following identities.
\begin{proposition}
\label{proposition:properties_Lie}
For~$\tau\in \Omega_1$, the Lie derivative~$\LL_\tau \gamma\colon \Omega_{n,p}\rightarrow \Omega_{n,p}$ satisfies
\begin{align}
\label{equation:vertical_Cartan}
\LL_\tau \gamma &= (d_V i_\tau +i_\tau d_V)\gamma, \quad \gamma\in \Omega_{n,p},\\
\LL_\tau \gamma &= (d_H \Liehomotopy_\tau +\Liehomotopy_\tau d_H)\gamma, \quad \gamma\in \Omega_{n}, \quad n>0, \nonumber\\
\LL_\tau \gamma &= (d_H \mathfrak{\Liehomotopy}_\tau +i_\tau \delta_V)\gamma, \quad \gamma\in \Omega_{0}, \nonumber
\end{align}
where~$\Liehomotopy_\tau=\LL_\tau h_H$ and~$\mathfrak{\Liehomotopy}_\tau=\LL_\tau \mathfrak{h}_H$ are the Lie homotopy operators.
\end{proposition}

\begin{proof}
As~$d_H$ and~$i_\tau$ commute, replacing~$d=(-1)^{n+p}d_H+d_V$ in~\eqref{equation:Cartan} gives~\eqref{equation:vertical_Cartan}. We then deduce from the expression~\eqref{equation:vertical_Cartan} that~$\LL_\tau \gamma\colon \Omega_{n,p}\rightarrow \Omega_{n,p}$.
For~$n>0$ and~$p=0$, the horizontal homotopy identity~\eqref{equation:horizontal_homotopy} yields
\[
\LL_\tau \gamma = i_\tau d_V\gamma
=i_\tau d_V d_H h_H \gamma+i_\tau d_V h_H d_H \gamma
=d_H \Liehomotopy_\tau \gamma+\Liehomotopy_\tau d_H \gamma,
\]
where we used that~$d_H$ commutes with~$i_\tau$ and~$d_V$ and that~$d_V$ commutes with~$h_H$.
Similarly, the augmented horizontal homotopy identity~\eqref{equation:augmented_horizontal_homotopy} yields the expression of~$\LL_\tau$ on~$\Omega_{0}$.
\end{proof}

\begin{remark}
The vertical homotopy operator is linked to the contraction operation by the identity~$i_{\atree1101}\gamma=\abs{\gamma}h_V\gamma$. Thus, equations~\eqref{equation:vertical_homotopy} and~\eqref{equation:vertical_Cartan} yield
\[\LL_{\atree1101} \gamma = \abs{\gamma}\gamma.\]
The equivalent of this property is used in~\cite{MuntheKaas18lbs} with planar forests to obtain the expansion of the Grossman-Larson exponential.
\end{remark}

The Lie derivative naturally realises a Lie algebra structure.
\begin{proposition}
\label{proposition:Lie_structure}
For~$\tau_1$,~$\tau_2\in \Omega_1$, define the commutator on aromatic vector fields by
$$\llbracket\tau_1,\tau_2\rrbracket:=\LL_{\tau_1}\tau_2-\LL_{\tau_2}\tau_1.$$
Then,~$(\Omega_1,\llbracket . , . \rrbracket)$ is a Lie algebra and~$\LL$ is a Lie algebra representation, that is, for~$\gamma\in \Omega_n$,
$$
[\LL_{\tau_1},\LL_{\tau_2}] \gamma :=\LL_{\tau_1}\LL_{\tau_2} \gamma-\LL_{\tau_2}\LL_{\tau_1} \gamma=\LL_{\llbracket\tau_1,\tau_2\rrbracket} \gamma.
$$
\end{proposition}

\begin{proof}
The Jacobi identity for the bracket~$\llbracket . , . \rrbracket$ is a consequence of the pre-Lie property of the Lie derivative~\cite{Saidi10oap,Manchon11lpl,Saidi11adh}, extended straightforwardly to aromatic trees.
Define for~$v\in V$ a node of~$\gamma$,~$\LL_{\tau}^v\gamma=i_\tau (\gamma_{v\rightarrow\Circled{1}})$. Then, we observe
\[
\LL_{\tau_1}\LL_{\tau_2} \gamma
=\sum_{\underset{v\neq w}{v, w\in V}} \LL_{\tau_1}^w\LL_{\tau_2}^v\gamma+ \LL_{\LL_{\tau_1}\tau_2}\gamma.
\]
As~$\LL_{\tau_1}^w$ and~$\LL_{\tau_2}^v$ commute, a similar expression for~$\LL_{\tau_2}\LL_{\tau_1} \gamma$ yields the result.
\end{proof}

Thanks to equation~\eqref{equation:vertical_Cartan}, the classical geometric properties of the Lie derivative extend to aromatic forms.
\begin{proposition}
\label{proposition:product_rule_Lie}
For~$\tau\in \Omega_1$, the Lie derivative~$\LL_\tau$ commutes with the horizontal and vertical derivatives~$d_H$ and~$d_V$. In particular, we have~$[\LL_\tau,d]=0$.
Moreover, for~$\mu\in\Omega_0$,~$\tau\in \Omega_1$, and~$\gamma\in \Omega_{n}$, the following product rule holds,
$$\LL_{\tau} (\mu\gamma)=\mu(\LL_{\tau} \gamma)+(\LL_{\tau} \mu)\gamma.$$
\end{proposition}

\begin{proof}
The Lie derivative commutes with~$d_H$ as~$d_H$ commutes with~$d_V$ and~$i_\tau$. Proposition~\eqref{proposition:derivatives_squared} yields
\[\LL_\tau d_V=d_V i_\tau d_V=d_V \LL_\tau.\]
The product rule is a consequence of the product rule for~$d_V$, that is,
$$d_V (\mu\gamma)=\mu(d_V \gamma)+(d_V \mu)\gamma.$$
Hence the result.
\end{proof}

The two Hopf algebra structures on standard B-series are associated to the composition and substitution laws~\cite{Connes98har,Chartier10aso,Calaque11tih} (see also~\cite{Bogfjellmo19aso,Bronasco22ebs,Rahm22aoa}).
The associated pre-Lie laws on~$\Omega_1$ are the grafting product, given for~$\tau$,~$\gamma\in \Omega_1$,~$r_\tau$ the root of~$\tau$, and~$V_\gamma$ the vertices of~$\gamma$ by
$$\tau \curvearrowright \gamma=\sum_{v\in V_{\gamma}} D^{r_\tau \rightarrow v} (\tau \gamma),$$
and the insertion product~\cite{Saidi10oap,Manchon11lpl,Saidi11adh}, that coincides with the Lie derivative~$\LL_\tau\gamma$.
We mention that the dual of the grafting product is also called Lie derivative (though it differs from~$\LL_\tau\gamma$) in~\cite{Hairer99bea} (see also~\cite[Sec.\ts IX.9.1]{Hairer06gni}) and is used for computing the modified equation of a B-series method in terms of trees.
The two pre-Lie structures interact according to the following identity, which is the pre-Lie version of the compatibility relation between the laws of composition and substitution of aromatic B-series.
For~$\tau_1$,~$\tau_2\in \Omega_1$ and~$\gamma\in \Omega_n$, we have
\begin{equation}
\label{equation:interaction}
\LL_{\tau_1}(\tau_2 \curvearrowright \gamma)=(\LL_{\tau_1} \tau_2)\curvearrowright \gamma+ \tau_2 \curvearrowright (\LL_{\tau_1} \gamma).
\end{equation}


\begin{remark}
Thanks to \cite{Floystad20tup}, it is known that $\Omega_1$ is not freely generated by the operations $\curvearrowright$ and $d_H$. For the elements of $\Omega_{n}$ that can be described by the operations $\curvearrowright$, $d_H$, and concatenation, an alternative expression of the Lie derivative is obtained by replacing in each term one node $\atree1101$ by $\tau$ in all possible ways.
For example, the aromatic Lagrangian $\atree2002$ satisfies
\[
\atree2002=d_H(\atree1101 \curvearrowright \atree1101)-\atree1101 \curvearrowright d_H\atree1101,\quad
\LL_\tau \atree2002=d_H(\tau \curvearrowright \atree1101) + d_H(\atree1101 \curvearrowright \tau) - \tau \curvearrowright d_H\atree1101-\atree1101 \curvearrowright d_H\tau.
\]
\end{remark}

The substitution law~$\triangleright$ of aromatic vector fields can be rewritten in terms of the Lie derivative in the spirit of~\cite{Oudom08otl} by using the pre-Lie structure of the Lie derivative~\cite{Saidi10oap,Manchon11lpl,Saidi11adh}.
We first extend~$\LL$ on~$\UU(\Omega_1)$, the universal enveloping algebra (that is, the symmetric tensor algebra) of~$\Omega_1$ over the base field~$\R$ equipped with the symmetric product~$\cdot$ and the shuffle coproduct~$\Delta$. Monomials in~$\UU(\Omega_1)$ are also called clumped forests in~\cite{Bronasco22cef} (see also~\cite{Bogfjellmo19aso}), and in particular we have in~$\UU(\Omega_1)$:
\[
(\atree1001 \atree2101)\cdot(\atree2001 \atree1101)=\frac{1}{2}(\atree1001 \atree2101)\otimes (\atree2001 \atree1101)
+\frac{1}{2}(\atree2001 \atree1101)\otimes (\atree1001 \atree2101)
\neq (\atree1001 \atree2001 \atree2101) \cdot \atree1101.
\]
We extend the Lie derivative for~$\tau_1$, \dots,~$\tau_n$,~$\gamma\in\Omega_1$ by
\[
\LL_{\tau_n \otimes \dots \otimes \tau_1} \gamma=\sum_{\underset{v_i\neq v_j}{v_1,\dots,v_n\in V}} \LL_{\tau_n}^{v_n} \dots \LL_{\tau_1}^{v_1}\gamma,
\]
where~$\LL_{\tau}^v=i_\tau (\gamma_{v\rightarrow\Circled{1}})$ and~$V$ is the set of vertices of~$\gamma$.
We then extend the Lie derivative~$\LL_\tau \gamma$ for~$\tau$,~$\gamma\in \UU(\Omega_1)$ using the Guin-Oudom process from~\cite[Prop.\ts 2.7]{Oudom08otl}: there exists a unique extension of the Lie derivative satisfying
\[
\LL_\ind \gamma_1=\gamma_1,\quad
\LL_{\tau\otimes\gamma_2} \gamma_1=\LL_\tau \LL_{\gamma_2} \gamma_1 - \LL_{\LL_\tau \gamma_2} \gamma_1,\quad
\LL_{\gamma_3} (\gamma_1\cdot \gamma_2)=(\LL_{\gamma_3^{(1)}}\gamma_1)\cdot (\LL_{\gamma_3^{(2)}}\gamma_2),
\]
where~$\tau\in \Omega_1$,~$\gamma_1$,~$\gamma_2$,~$\gamma_3\in\UU(\Omega_1)$ and we use Sweedler's notation.
This allows us to rewrite the substitution in terms of the Lie derivative and the concatenation exponential. Note that equation~\eqref{equation:def_perturbation_Lie_der} is straightforwardly derived from the following result.
\begin{proposition}
\label{proposition:expansion_pre_Lie_substitution}
Let~$\tau$,~$\gamma\in \Omega_1$ and~$\varepsilon>0$, then the perturbation of~$\gamma$ by~$\varepsilon\tau$ is
\[
(\atree1101+\varepsilon \tau)\triangleright \gamma
=\LL_{\exp(\varepsilon \tau)} \gamma,\quad
\exp(\varepsilon \tau)=\ind +\varepsilon \tau+\frac{\varepsilon^2}{2}\tau\cdot\tau+\frac{\varepsilon^3}{3!}\tau\cdot\tau\cdot\tau+\dots
\]
\end{proposition}

\begin{proof}
Let~$\tau\in \Omega_1$,~$\gamma\in \FF_1$, then the substitution rewrites by definition as
\[
\tau\triangleright \gamma
=\frac{1}{N!} \LL_{\tau^{N}} \gamma
=\frac{1}{N!} \LL_{\tau^{\cdot N}} \gamma,
\]
where~$N$ is the number of vertices of~$\gamma$.
As~$\atree1101$ is the neutral element for substitution, we find
\[
(\atree1101+\varepsilon \tau)\triangleright \gamma
=\frac{1}{N!} \LL_{(\atree1101+\varepsilon \tau)^{\cdot N}} \gamma
=\frac{1}{N!}\sum_{k=0}^N \binom{N}{k} \varepsilon^k \LL_{\atree1101^{\cdot N-k} \cdot \tau^{\cdot k}} \gamma.
=\frac{1}{k!}\sum_{k=0}^N \varepsilon^k \LL_{\tau^{\cdot k}} \gamma.
\]
As~$\LL_{\tau^{\cdot M}} \gamma=0$ if~$M>N$, we obtain the result by linearity.
\end{proof}

\subsection{Symmetries and divergence symmetries}

The Lie derivative allows us to define the concepts of symmetries and divergence symmetries.
\begin{definition}
The aromatic vector field~$\tau\in \Omega_1$ is a symmetry for the aromatic form~$\gamma$ if~$\LL_{\tau} \gamma=0$. If there exists an aromatic form~$\eta$ such that~$\LL_{\tau} \gamma=d_H \eta$, we say that~$\tau$ is a divergence symmetry for~$\gamma$.
\end{definition}

The simplest symmetries are also called solenoidal~\cite{Laurent23tab}. They are the object of ultimate interest in numerical volume-preservation.
\begin{definition}
An aromatic vector field~$\tau\in \Omega_1$ is solenoidal if it is a symmetry for the aromatic Lagrangian~$\atree1001$.
\end{definition}

Thanks to the exactness of the aromatic bicomplex (see Theorem~\ref{theorem:exactness}), the set of solenoidal forms is~$d_H (\Omega_2)$. In particular, the simplest examples of solenoidal forms are
\begin{align*}
2d_H\wedge \atree1101 \atree2101 &=\atree2002 \atree1101+\atree2001 \atree1101-\atree1001 \atree2101-\atree3101,\\
2d_H\wedge \atree1101 \atree3102 &=\atree3001 \atree1101+\atree3002 \atree1101+\atree3003 \atree1101-\atree4103-\atree4104-\atree1001 \atree3102,\\
2d_H\wedge \atree1101 \atree3101 &=\atree3004 \atree1101+2\atree3002 \atree1101+\atree4103-2\atree4104-\atree4101-\atree1001 \atree3101,\\
2d_H\wedge \atree1001 \atree1101 \atree2101 &=\atree3001 \atree1101+\atree1001 \atree2002 \atree1101+\atree1001 \atree2001 \atree1101-\atree2001 \atree2101-\atree1001 \atree1001 \atree2101-\atree1001 \atree3101.
\end{align*}
As the complexity of the calculation increases rapidly with the order, there is no known example of symmetry that is not solenoidal at the present time.

Let us state the aromatic formulation of the first variational formula, a central tool in the proof of Noether's theorem.
\begin{proposition}[First variational formula]
\label{lemma:link_dV_deltaV}
Let~$\gamma\in \Omega_0$ (respectively~$\gamma\in \II_p$), then there exists an aromatic form~$\eta\in \Omega_{1,1}$ (respectively~$\eta\in \Omega_{1,p+1}$) such that
\begin{equation}
\label{equation:first_variation_formula}
d_V \gamma=\delta_V \gamma+d_H \eta.
\end{equation}
In particular, let~$\tau\in \Omega_{1}$ and~$\gamma\in \Omega_0$, then there exists~$\eta\in\Omega_{1}$ such that
\begin{equation}
\label{equation:first_variation_formula_Lie}
\LL_{\tau} \gamma=i_\tau \delta_V \gamma+d_H\eta.
\end{equation}
\end{proposition}

\begin{proof}
Let~$\gamma\in \Omega_0$ or~$\gamma\in \II_p$ and let~$\omega=d_V\gamma$. As~$\omega-I\omega\in \Ker(I)$, by horizontal exactness of the aromatic bicomplex (Theorem~\ref{theorem:exactness}), there exists~$\eta$ such that
$$d_V\gamma=\omega=I\omega +(\omega-I\omega)=I\omega+d_H\eta.$$
Applying the contraction~$i_\tau$ to~\eqref{equation:first_variation_formula} yields~\eqref{equation:first_variation_formula_Lie} as~$i_\tau$ and~$d_H$ commute.
\end{proof}

A first application of the first variational formula, in particular of equation~\eqref{equation:first_variation_formula_Lie}, is that any~$\tau\in \Omega_1$ is a divergence symmetry for~$\gamma\in \Ker(\delta_V|_{\Omega_0})=d_H(\Omega_1)$. In particular, solenoidal forms (and more generally symmetries) are divergence symmetries.
The elements with up to three nodes in~$d_H(\Omega_1)$ are spanned by
\begin{align*}
\{
&\atree1001,
\atree2001+\atree2002,
\atree2001+\atree1001\atree1001,
\atree3001+\atree3002+\atree3003,
\atree3004+2\atree3002,\\
&\atree3001+\atree2001 \atree1001+\atree2002 \atree1001,
\atree3004-\atree2002 \atree1001,
\atree1001 \atree1001 \atree1001+2\atree2001 \atree1001\}
\end{align*}

An alternative way to produce symmetries is given by the followingresult, a corollary of Propositions~\ref{proposition:Lie_structure} and~\ref{proposition:product_rule_Lie}.
\begin{proposition}
If~$\tau_1\in \Omega_1$ is a divergence symmetry for~$\gamma\in \Omega_0$, then for any~$\tau_2\in \Omega_1$,~$\llbracket\tau_1,\tau_2\rrbracket$ is a divergence symmetry for~$\gamma$.
If~$\tau_1$ is a symmetry for~$\gamma$, then~$\tau_1$ is a symmetry for~$\LL_{\tau_2} \gamma$ if and only if~$\llbracket\tau_1,\tau_2\rrbracket$ is a symmetry for~$\gamma$.
\end{proposition}

\begin{remark}
As the number of aromatic trees grows fast with the number of nodes, one is often interested in finding symmetries for specific differential systems. Let~$f\colon \R^d\rightarrow\R^d$ be a smooth vector field, and~$F$ be the elementary differential map (see~\cite{Laurent23tab}). We call~$f$-symmetry for~$\gamma$ an aromatic vector field~$\tau\in \Omega_1$ such that~$F(\LL_\tau \gamma)(f)=0$, and~$\tau$ is~$f$-solenoidal in the specific case where~$\gamma=\atree1001$.
In particular, if we assume that~$\Div(f)=0$, then~$\tau=\atree2002 \atree1101-\atree3101$ is a~$f$-symmetry for~$\atree1001$.
The assumption~$\Div(f)=0$ does not create new~$f$-solenoidal vector fields~\cite{Laurent23tab}, but it does create new~$f$-symmetries.
In particular, a~$f$-solenoidal aromatic vector field with~$\Div(f)=0$ is a~$f$-symmetry of any aroma~$\gamma\in \Omega_0$ with a 1-loop.
For example,~$\tau=\atree2002 \atree1101-\atree3101$ is a~$f$-symmetry for~$\atree1001$,~$\atree2001$,~$\atree3004$,...
The use of specific vector fields~$f$ gives rise to degeneracies, which makes it easier to find~$f$-symmetries. We cite in particular the work~\cite[Sec.\ts 4]{Bogfjellmo22uat} that uses such degeneracies with~$f$ quadratic.
%
%
%
%
\end{remark}

\subsection{The aromatic Noether theorem}

The Noether theorem, published in the paper Invariante Variationsprobleme by Emmy Noether in 1918 (see the english translation~\cite{Noether71ivp}), draws an explicit link between conservation laws and symmetries in the context of variational calculus.
\begin{definition}
The aromatic vector field~$\tau\in \Omega_{1}$ is a generator of a conservation law~$\eta\in\Omega_{1}$ for the source form~$\gamma\in \II_1$ if~$i_{\tau}\gamma=d_H\eta$.
\end{definition}

We rewrite Noether's theorem in the context of aromatic forms.
\begin{theorem}[Noether's theorem]
\label{theorem_global_Noether}
Consider an aromatic Lagrangian~$\gamma\in \Omega_0$ and an aromatic vector field~$\tau\in \Omega_{1}$. Then,~$\tau$ is a divergence symmetry of~$\gamma$ if and only if~$\tau$ is the generator of a conservation law for~$\delta_V\gamma$.
\end{theorem}

Theorem~\ref{theorem_global_Noether} is a direct consequence of the first variational formula~\eqref{equation:first_variation_formula_Lie}. In the context of variational calculus, a symmetry of the Lagrangian exactly corresponds to the preservation of a quantity, which typically leads to superfluous degrees of freedom (see, for instance, the textbooks~\cite{Anderson89tvb,Arnold89mmo,Arnold06mao}).

\begin{ex*}
Let~$\gamma\in d_H(\Omega_1)$, then as~$d_H$ and~$\LL_\tau$ commute, any aromatic vector field~$\tau$ is a divergence symmetry for~$\gamma$. The associated conservation law is~$\eta=0$.
\end{ex*}

\begin{remark}
\label{remark:local_Noether}
An alternate formulation of Theorem~\ref{theorem_global_Noether} with source forms is the following. Define the natural Lie derivative on~$\II_p$ by~$\LL^\natural_\tau=I\LL_\tau$.
Assume that~$\gamma\in \II_1$ satisfies~$\delta_V \gamma=0$ (i.e.\ts,~$\gamma$ is variational), then the aromatic vector field~$\tau\in \Omega_1$ is a symmetry for~$\gamma$ if and only if~$\delta_V i_\tau \gamma=0$.
This formulation of the Noether theorem is a consequence of the first variational formula~\eqref{equation:first_variation_formula_Lie} and the identities
$$\LL^\natural_\tau \gamma = (\delta_V i_\tau +I i_\tau \delta_V)\gamma, \quad \delta_V \LL_\tau \gamma=\LL^\natural_\tau d_V \gamma.$$
\end{remark}

%
%

\subsection{Application to numerical volume-preservation}
\label{section:application_VP}

The search for a volume-preserving integrator in the form of an aromatic B-series method, or equivalently of an affine-equivariant method~\cite{McLachlan16bsm,MuntheKaas16abs,Laurent23tue}, is an important open question of geometric numerical integration. It is known that there is no volume-preserving B-series method except the exact flow~\cite{Chartier07pfi,Iserles07bsm}, but the question for aromatic B-series methods is still open. We refer to the recent works~\cite{Bogfjellmo19aso,Bogfjellmo22uat,Celledoni22dad,Laurent23tab} for different approaches to the conjecture.
In this subsection, we apply the results on the Lie derivative, the symmetries, and the Noether theorem in the context of volume-preservation, and we rewrite the numerical open questions in a purely algebraic manner.

We denote~$\overline{\Omega}_1$ the set of formal series of elements of~$\Omega_1$ graded by the number of nodes of the aromatic trees, also called aromatic B-series. The previous results presented in this paper extend straightforwardly to aromatic B-series.
Thanks to~\cite[Thm.\ts 4.17]{Laurent23tab}, the aromatic B-series of a consistent volume-preserving integrator essentially\footnote{In the numerical context, the vector field of the ordinary differential equation of interest is assumed to be divergence-free, so that the aromas with a node linked to itself, such as~$\atree1001$, do not appear in the formal series. It was shown in~\cite{Laurent23tab} that adding this degeneracy condition does not create new solenoidal forms, so that we do not lose any generality by working on~$\overline{\Omega}_1$ here.} takes the form of the following substitution~$(\atree1101+\tau)\triangleright e$, where~$\tau\in\overline{\Omega}_1$ is a formal series of solenoidal forms and~$e\in\overline{\Omega}_1$ is the standard B-series of the exact flow of~$y'=f(y)$ (see~\cite[Chap.\ts III]{Hairer06gni}).
According to Proposition~\ref{proposition:expansion_pre_Lie_substitution}, the problem of volume-preservation then boils down to finding an aromatic B-series method (or class of methods) that has the form \eqref{eqaution:expansion_substitution}.
Note that choosing~$\tau=0$ in equation~\eqref{eqaution:expansion_substitution} yields the exact flow~$e$, which is the only volume-preserving B-series method~\cite{Chartier07pfi,Iserles07bsm}.
\begin{align}
\label{eqaution:expansion_substitution}
(\atree1101+\tau)\triangleright e&=\LL_{\exp(\tau)} e \\
&=e
+\LL_\tau e
+\frac{1}{2}(\LL_\tau^2-\LL_{\LL_\tau\tau})e \nonumber\\
&+\frac{1}{6}(\LL_\tau^3-2\LL_\tau \LL_{\LL_\tau\tau}-\LL_{\LL_\tau\tau}\LL_\tau +\LL_{\LL_{\LL_\tau\tau}\tau} +\LL_{\LL_\tau^2\tau} ) e
+ \dots \nonumber
\end{align}

\begin{remark}
One is also interested in finding numerical methods that preserve modified measures. For quadratic ODEs, the Kahan-Hirota-Kimura discretization~\cite{Kahan93unm,Hirota00dote,Hirota00dotl} can preserve some modified measures~\cite{Celledoni22dad, Bogfjellmo22uat}.
In this context, one searches for aromatic Lagrangians~$\mu\in\Omega_0$ and modified vector field represented by~$\tau\in \Omega_1$ such that~$(1+\mu)\tau$ is solenoidal.
\end{remark}

Let us adapt Noether's theory in the context of volume-preservation.
While Theorem~\ref{theorem_global_Noether} and Remark~\ref{remark:local_Noether} focus on Lagrangians in~$\Omega_0$ and source forms in~$\II_1$ in variational calculus, we are also interested in symmetries and conservation laws on~$\Omega_1$ and~$\Omega_2$.
The following result, derived from Theorem~\ref{theorem:exactness}, gives necessary conditions on the form of the modified vector field of a volume-preserving method. We recall that as the aromatic bicomplex is exact, any solenoidal form~$\gamma\in \Omega_1$ is the image of a form in~$\hat{\gamma}\in \Omega_2$, that is,~$\gamma=d_H \hat{\gamma}$.
\begin{theorem}
\label{theorem_global_Noether_vp}
Let~$\gamma=d_H \hat{\gamma}\in \Omega_1$ be solenoidal, then every aromatic vector field~$\tau\in \Omega_1$ is a divergence symmetry for~$\gamma$.
Moreover,~$\tau$ is a divergence symmetry of~$\hat{\gamma}$ (and thus a symmetry of~$\gamma$) if and only if there exists an aromatic form~$\eta\in \Omega_3$ such that
$
H_\tau \gamma=d_H \eta.
$
\end{theorem}


%
%


%

There is no known non-trivial modified vector field of a volume-preserving method that has symmetries to the best of our knowledge.
A better understanding of the Lie derivative, the symmetries, and especially the solenoidal vector fields could give insight on the form of a volume-preserving method and help describe the degrees of freedom we have in the choice of the method.
%
There exists a vast literature on the description of symmetries in variational calculus, which further motivates collaborations on the application of the tools of variational calculus for the creation of volume-preserving integrators. In addition, extending the aromatic bicomplex and the aromatic Noether's theory for the study of volume-preservation on manifolds or for the exact numerical preservation of the invariant measure of ergodic stochastic systems is matter for future work.

\bigskip

\noindent \textbf{Acknowledgements.}\
The work of the author was supported by the Research Council of Norway through project 302831 ``Computational Dynamics and Stochastics on Manifolds'' (CODYSMA).
The author would like to thank E.\ts Bronasco, D.\ts Manchon, H.\ts Munthe-Kaas, and the participants of the MAGIC workshops for the enlightening discussions we shared.

\bibliographystyle{abbrv}
\bibliography{Ma_Bibliographie}

\end{document}